\documentclass[11pt]{article}

\usepackage{amsmath, amsfonts, amsthm, verbatim,
amssymb, dsfont, xcolor}

\newcommand{\dd}{\mathrm{d}}
\newcommand{\R}{\mathbb{R}}

\newcommand{\E}{\mathbf{E}}
\newcommand{\p}{\mathbf{P}}
\newcommand{\ii}{\mathfrak{i}}

\newcommand{\bone}{\mathds 1}
\newcommand{\dint}{\int \hspace*{-5pt} \int}

\theoremstyle{plain}
\newtheorem{lemma}{Lemma}
\newtheorem{theorem}{Theorem}
\newtheorem{proposition}{Proposition}
\newtheorem{corollary}{Corollary}

\theoremstyle{remark}

\newtheorem{remark}{Remark}

\usepackage{authblk}

\title{Tail behavior and almost sure growth rate of supOU processes}
\author[1]{Danijel Grahovac\thanks{dgrahova@mathos.hr}}
\author[2]{P\'eter Kevei\thanks{kevei@math.u-szeged.hu}}
\affil[1]{School of Applied Mathematics and Informatics, J. J. Strossmayer University of Osijek, Trg Ljudevita Gaja 6, 31000 Osijek, Croatia}
\affil[2]{Bolyai Institute, University of Szeged, Aradi v\'ertan\'uk tere 1, 6720 Szeged, Hungary}

\date{}

\begin{document}

\maketitle

\begin{abstract}
In this paper we consider sample path growth of superpositions of  
Ornstein--Uhlenbeck type processes (supOU). SupOU processes
are stationary infinitely divisible 
processes defined as integrals with respect to a random measure. They allow marginal 
distributions and correlations to be modeled independently. Our results show that the 
almost sure behavior is primarily governed by the tail of the marginal distribution. In 
particular, we obtain a general integral test for the sample path growth that covers both 
heavy-tailed and light-tailed scenarios. We also investigate the tail behavior of the 
marginal distributions in connection with the characteristics of the underlying random 
measure.

\textit{AMS 2020 Subject Classiﬁcations:} 60G17, 60F15, 60G55, 60G57

\end{abstract}

\section{Introduction}

SupOU processes are strictly stationary processes defined as integrals
\begin{equation} \label{eq:supOU-def}
X(t) = \dint_{(0,\infty) \times (-\infty,t]} e^{-x(t-s)} \Lambda(\dd x, \dd s),
\end{equation}
with respect to an infinitely divisible independently scattered random measure $\Lambda$ given by
\begin{equation} \label{eq:Lambda}
\Lambda(\dd x, \dd s) = 
m \, \pi(\dd x) \dd s + 
\int_{(0,1]} z  (\mu - \nu) ( \dd x, \dd s, \dd z) + \int_{(1,\infty)} z  \mu ( \dd x, \dd s, \dd z),
\end{equation}
where $m \in \R$ and $\mu$ is a Poisson random measure on 
$(0,\infty) \times \R \times (0,\infty)$
with intensity measure 
\begin{equation*}
\nu (\dd x, \dd s, \dd z) = \pi(\dd x) \dd s \, \lambda(\dd z).
\end{equation*}
Here $\pi$ is a measure on $(0,\infty)$, such that $\int x^{-1} \pi(\dd x) < \infty$, and 
$\lambda$ is a L\'evy measure (i.e.~$\int_{(0,\infty)} ( 1 \wedge z^2) \lambda(\dd z) < 
\infty$) satisfying $\int_{(1,\infty)} \log z \, \lambda (\dd z ) < \infty$. We note that the random measure $\Lambda$ in \eqref{eq:Lambda}  may also include a Gaussian component and 
$\lambda$ may be supported on the 
whole $\R$. For notational ease, here we focus on supOU processes with 
$\lambda$ allowing positive jumps only. However, all our results remain 
true for measures $\lambda$ allowing negative jumps as well.

This class of processes was introduced in \cite{Barn}
(with different parametrization, see Section \ref{sect:as})
and generalized to the multivariate case in \cite{barndorff2011multivariate}. 
SupOU processes have proved to be a flexible model allowing marginal 
distribution and dependence structure to be defined separately. Namely, the 
one-dimensional marginal distribution is determined by $m$ and the L\'evy measure $\lambda$. It is self-decomposable and corresponds to the marginal distribution of a L\'evy-driven Ornstein--Uhlenbeck process
\begin{equation}\label{eq:OU}
\int_{(-\infty,t]} e^{-x(t-s)} L(\dd s),
\end{equation}
where $L$ is a two-sided L\'evy process with L\'evy--Khintchine triplet 
$(m,0,\lambda)$ and $x>0$ (see e.g.~\cite{apple}). The superposition is obtained by averaging over randomized mean-reverting 
parameter $x$ following 
distribution $\pi$, if $\pi$ is indeed a probability measure. 
While the OU process in \eqref{eq:OU} has correlation function 
$e^{-xt}$, the supOU process has a correlation function 
$c\int_{(0,\infty)} x^{-1} e^{-x t} \pi(\dd x)$, assuming that both correlation
functions exist. Hence, the choice of $\pi$ provides
various forms of dependence, although always with monotonically decreasing correlations. 

Due to their flexibility, supOU processes have many applications. 
Since they cover most of the stylized facts typically encountered 
in financial data, 
positive supOU processes provide a plausible model for stochastic 
volatility. The properties of supOU stochastic volatility models have been 
extensively studied in 
\cite{bn2013supOUSV,curato2019,fuchs2013,moser2011,stelzer2015derivative}. SupOU 
processes have also been used in astrophysics for assessing the mass of black holes 
\cite{kelly2011stochastic,kelly2013active}. 

The extremal behavior of supOU processes was analyzed in \cite{FasenKlupp} based on 
the general results of \cite{Fasen05}. 
L\'evy-driven spatial random fields with regularly varying L\'evy measure were 
treated in \cite{rm2022}. Ergodic properties of 
supOU processes are established in \cite{fuchs2013}. For integrated supOU processes, limit 
theorems have been investigated in \cite{GLST2019,GLT2019Limit,GLT2020multifaceted}.
All these papers concern distributional properties. In \cite{basse2020sufficient}, a sufficient condition is provided for the existence of
c{\`a}dl{\`a}g sample paths of stable supOU processes (see also \cite{MR3449388}).

In this paper we derive almost sure growth rates of the paths of supOU processes. 
Almost sure growth rates of stochastic integrals with respect to a L\'evy 
basis were investigated in \cite{CK, CK22} for the solution to the stochastic 
heat equation, and in the companion paper \cite{GK2} for the integrated 
supOU process.

In Section \ref{sect:tail} we analyze the marginal distributions of $(X(t))$.
We recall some known results on the regular variation of the tail, 
and we prove a generalization to dominatedly varying distributions. 
In particular, in Lemma \ref{lemma:domvar} we obtain the tail asymptotics 
of the L\'evy measure of $X$ in terms of the jump measure $\lambda$, 
extending the results in \cite{FasenKlupp}. This tail estimate plays a key role in 
determining the growth rate for heavy-tailed $\lambda$, see Corollary 
\ref{cor:supOU-as-D}. 

The main results of the paper are contained in 
Section \ref{sect:as}, where we turn to the almost sure path behavior. In contrast to Section \ref{sect:tail}, here we focus only on supOU processes with jump measure 
satisfying $\int_{(0,1]}z\lambda (\dd z)<\infty$. 
Proposition \ref{prop:largexi} shows that if $\pi$ is 
an infinite measure and $\lambda$ has unbounded support, then in any open 
interval $X(t)$ is unbounded. In our main result, Theorem \ref{thm:as},
we obtain a general integral test allowing both heavy-tailed and light-tailed
jump measure $\lambda$. As a consequence, we show that for any normalizing function $f$ and heavy-tailed L\'evy measures, $\limsup_{t \to \infty} X(t) / f(t)$
is either 0, or $\infty$, that is, as usual for heavy-tailed processes, 
there is no proper normalization almost surely.
In the light-tailed scenario we show that for bounded jump measures $\lambda$
the almost sure growth rate is $\log t / (\log \log t)$ 
(Theorem \ref{thm:as-boundedlambda}),  while for measure 
with exponential tail, the rate is $\log t$ (Corollary \ref{cor:exp}).

As noted before, all our results remain true if $\lambda$ is supported on the 
whole $\R$, since positive and negative jumps are independent, and so are the corresponding 
processes; see Remark \ref{remark:pos-neg} after Theorem \ref{thm:as}. 
Similarly, if there is a Gaussian component in \eqref{eq:Lambda}, the 
corresponding Gaussian supOU is independent of the non-Gaussian part.
Almost sure growth of Gaussian processes, i.e.~law of iterated logarithm,
is well-understood, see e.g.~\cite{orey1972, watanabe1970}.

\section{Tail behavior and moments} \label{sect:tail}

For any measure $Q$ on $(0,\infty)$ introduce the notation 
\begin{equation*}
m_p ( Q ) = \int_{(0,\infty)} x^{p} Q (\dd x),
\end{equation*}
and for the tail we put $\overline Q(x) = Q((x,\infty))$, $x > 0$.

By Theorem 3.1 in \cite{Barn}
\[
\begin{split}
\E e^{\ii \theta X(t)} & = 
\exp \left\{ \ii \theta a + 
\int_{(0,\infty)} 
\left( e^{\ii \theta y} - 1 - \ii \theta y \bone(y \leq 1) \right) 
\eta (\dd y) \right\},
\end{split}
\]
where the L\'evy measure $\eta$ is defined as
\[
\eta(B) = m_{-1}(\pi) \,
\mathrm{Leb} \times \lambda \left( 
\{ (u,z) : e^{-u} z \in B,  u > 0, z > 0 \}  \right),
\]
and $a \in \R$ is an explicit constant. 
Recall that marginal distributions of supOU coincide with the distribution
of the corresponding L\'evy-driven OU in \eqref{eq:OU}.
The tail of $\eta$ can be easily expressed with the tail of $\lambda$ as
\begin{equation} \label{eq:eta-bar}
\begin{split}
\overline \eta(r) & = \eta( (r,\infty) ) = 
m_{-1}(\pi) \dint \bone( e^{-u} z > r) \dd u \lambda(\dd z) \\
& = m_{-1}(\pi) \int_0^\infty \overline \lambda(r e^u) \dd u
= m_{-1}(\pi) \int_{r}^\infty \frac{\overline \lambda(z)}{z } \dd z .
\end{split}
\end{equation}

This simple relation of the two tails has immediate consequences.
Let $\mathcal{RV}_{\alpha}$ denote the class of regularly varying functions
with index $\alpha$.
As a consequence of \eqref{eq:eta-bar}, the tail of $\lambda$
is regularly varying if and only if the tail of $\eta$ is 
regularly varying. More precisely, 
it was shown in Proposition 3.1 in \cite{FasenKlupp} 
(assuming also $m_0(\pi) < \infty$, which is irrelevant here)
that
$\overline \lambda \in \mathcal{RV}_{-\gamma}$ for some $\gamma > 0$
if and only if $\overline \eta \in \mathcal{RV}_{-\gamma}$, and 
in this case
\[
\overline \eta (r ) \sim m_{-1}(\pi) \frac{1}{\gamma} \overline \lambda(r), 
\quad r \to \infty,
\]
where $\sim$ stands for asymptotic equivalence, that is
$f(x) \sim g(x)$ as $x \to \infty$, if $\lim_{x \to \infty} \tfrac{f(x)}{g(x)} = 1$.
For the same result but in a much broader context
see \cite[Example 5.2]{Jacobsen}, where it is also shown that 
the same equivalence holds in case of double-sided Ornstein--Uhlenbeck processes.

The tail behavior of the supremum was determined in
\cite[Proposition 3.2]{Fasen05}, see also \cite[Proposition 3.3]{FasenKlupp}.
Under the assumption $m_0(\pi) < \infty$ (and here it is relevant, see 
Proposition \ref{prop:largexi}) 
it was shown in \cite{Fasen05} that if
$\overline \lambda \in \mathcal{RV}_{-\gamma}$ for some $\gamma > 0$, 
then, for any $T > 0$,  as $r \to \infty$
\begin{equation}\label{eq:supXlambda}
\p \Big( \sup_{t \in [0,T]} X(t) > r \Big) \sim 
(T m_{0}(\pi) + \gamma^{-1} m_{-1}(\pi) ) \overline \lambda(r).
\end{equation}

For $\gamma = 0$, that is if $\overline \lambda$ is slowly varying, Proposition 1.5.9b in \cite{BGT} implies that $\overline \eta 
\in \mathcal{RV}_0$, and as $r \to \infty$
\[
\frac{\overline \eta(r)}{\overline \lambda(r)} \to \infty.
\]
However, the converse does not hold, i.e.~$\overline \eta \in \mathcal{RV}_0$
does not imply that $\overline \lambda \in \mathcal{RV}_0$, see the 
remark after Corollary 8.1.7 in \cite{BGT}, or Theorem 1.1 in \cite{Kevei}.

Regularly varying tails are subexponential, therefore in this case 
the tail of the infinitely divisible random variable $X(t)$ is the 
same as the tail of its L\'evy measure (see e.g.~\cite{EGV}), so
\begin{equation} \label{eq:X-eta-lambda}
\p ( X(t) > r) \sim \overline \eta (r) \sim 
 m_{-1}(\pi) \frac{1}{\gamma} \overline \lambda(r) \quad \text{as } r \to \infty.
\end{equation}
In particular, by \eqref{eq:supXlambda} 
the tail of $X(t)$ and $\sup_{t \in [0,1]} X(t)$ are 
asymptotically the same up to a constant factor.

However, to estimate growth rate much weaker conditions are sufficient.
Therefore, we turn to more general tail behavior. 
A nonincreasing function $f$ is \emph{dominatedly varying}, $f \in \mathcal{D}$, if
\begin{equation*} 
\limsup_{x \to \infty} \frac{f(x)}{f(2x)} < \infty;
\end{equation*}
see \cite[Section 1.10]{BGT}, \cite{ShimWat}.
The result in \cite{EGV} was extended in \cite[Proposition 4.1]{Watanabe} 
to dominatedly varying tails, showing that the ratio of the 
tails is bounded away from 0 and $\infty$. Next 
we show that dominated variation of $\overline \lambda$ implies the same 
for $\overline \eta$. This is a generalization of Proposition 3.1 in
\cite{FasenKlupp}.

\begin{lemma} \label{lemma:domvar}
Assume that $\overline \lambda \in \mathcal{D}$. Then $\overline \eta \in \mathcal{D}$,
and 
\begin{equation} \label{eq:lambda/eta}
\limsup_{r \to \infty} \frac{\overline \lambda(r)}{\overline \eta(r) } < \infty.
\end{equation}
Furthermore, if for some $c > 1$
\begin{equation} \label{eq:lambda-growth}
\limsup_{r \to \infty} 
\frac{\overline \lambda(c r)}{\overline \lambda(r)} < 1,
\end{equation}
then
\[
\limsup_{r \to \infty} \frac{\overline \eta(r) }{\overline \lambda(r)} < \infty.
\]
\end{lemma}

\begin{proof}
Since $\overline \lambda \in \mathcal{D}$, there exists $c_1 > 0, t_0 > 0$, such
that for $t > t_0$
\begin{equation*} 
\overline \lambda(t) \leq c_1 \overline \lambda(2t). 
\end{equation*}
Thus, for $t > t_0$
\[
\begin{split}
\overline \eta (t) & = 
m_{-1}(\pi)
\int_t^\infty \frac{\overline \lambda(z)}{z} \dd z 
= 
m_{-1}(\pi) \int_t^{2t} \frac{\overline \lambda(z)}{z} \dd z + \overline \eta(2t) \\
& \leq m_{-1}(\pi) c_1 \int_{2t}^{4t} \frac{\overline \lambda(y)}{y} \dd y + \overline 
\eta(2t) 
\leq (c_1 + 1) \overline \eta(2t),
\end{split}
\]
showing that $\overline \eta \in \mathcal{D}$. 

Similarly, for $t \geq t_0$
\[
\begin{split}
\overline \eta(t) & \geq m_{-1}(\pi) 
\int_t^{2t}  \frac{\overline \lambda(z)}{z} \dd z 
\geq m_{-1}(\pi) \overline \lambda(2t) \log 2 
\geq m_{-1}(\pi) \frac{\log 2}{c_1} \overline 
\lambda(t),
\end{split}
\]
proving \eqref{eq:lambda/eta}.

Assume \eqref{eq:lambda-growth}.  
Then, for $t$ large enough, for some $\delta > 0$ and $k = 0,1,\ldots$
\[
\begin{split}
\overline \eta( c^k t) - \overline \eta( c^{k+1} t) & = 
m_{-1}(\pi)
\int_{c^k t}^{c^{k+1} t} \frac{\overline \lambda (z)}{z} \dd z \\
& \leq 
m_{-1}(\pi) \overline \lambda( c^k t) \log c 
\leq m_{-1}(\pi) \overline \lambda (t)(1-\delta)^{k} \log c.
\end{split}
\]
Thus
\[
\overline \eta(t) \leq m_{-1}(\pi) \overline \lambda(t) \frac{\log c}{\delta},
\]
and the proof is complete.
\end{proof}

We note that if $\lambda$ has exponential, or lighter tail, then
\eqref{eq:lambda/eta} does not hold. Roughly speaking, if 
$\lambda$ has light tail, then $\eta$ has even lighter, 
if $\lambda$ has a power law tail, then $\eta$ has it too, and they 
are of the same order, while  
if $\lambda$ has a slowly varying tail, then $\eta$ has 
an even heavier slowly varying tail.

\medskip

From the tail behavior we can obtain necessary and 
sufficient conditions for the existence of moments.
A function $g$ is submultiplicative, if there exists $a > 0$ such 
that for any $x,y$, $g(x+y) \leq a g(x) g(y)$.
For a submultiplicative function $g$,
$\E g(X(t)) < \infty$ if and only if $\int_{(1,\infty)} g(x)  \eta (\dd x) < \infty$
(see e.g.~\cite[Theorem 25.3]{Sato}). Since
\begin{equation} \label{eq:eta-integral}
\begin{split}
\int_{(1,\infty)} g(x) \eta(\dd x)
& = m_{-1}(\pi) \int_{(0,\infty)} \int_{0}^\infty 
g(e^{-u} z) \bone (e^{-u} z > 1 ) \dd u \lambda(\dd z) \\
& = m_{-1}(\pi) \int_{(1,\infty)} \int_1^z g(y) y^{-1} \dd y \, \lambda(\dd z)  \\
& = m_{-1}(\pi) \int_1^\infty g(y) y^{-1} \, \overline \lambda(y) \, \dd y,
\end{split}
\end{equation}
we obtain the following.

\begin{corollary} \label{cor:Xmoments}
Let $g$ be a nonnegative submultiplicative function. Then
$\E g(X(t)) < \infty$ if and only if 
$\int_{1}^\infty g(y) y^{-1} \overline \lambda(y) \dd y < \infty$.
In particular, for any $\beta > 0$,
$\E X(t)^\beta < \infty$ if and only if 
$\int_{(1,\infty)} z^\beta \lambda(\dd z) < 
\infty$.
\end{corollary}

\section{Almost sure growth} \label{sect:as}

In the original paper \cite{Barn}, 
supOU processes were introduced with a slightly different
parametrization, with $-x t + s$ in the exponent in \eqref{eq:supOU-def}. 
Representation \eqref{eq:supOU-def} is a moving 
average representation that was used in \cite{Fasen05} and 
\cite{barndorff2011multivariate} (see also \cite{FasenKlupp}). 
Necessary and sufficient conditions for the existence of $X(t)$ are  
the following:
(i) $\int_{[1,\infty)} \log z \, \lambda(\dd z ) <  \infty$ and 
(ii) $m_{-1} (\pi) = \int_{(0,\infty)} x^{-1} \pi(\dd x) < \infty$. 
In particular, $\pi$ need not be a finite (probability) measure. However, the 
latter, i.e.~$m_0(\pi) < \infty$, was assumed in \cite{Fasen05, FasenKlupp}.

In all the results of this section we assume that 
$\int_{(0,1]} z \lambda (\dd z) < \infty$. 
Since the drift parameter $m$ in \eqref{eq:Lambda} only adds a deterministic 
constant, we further assume that 
$m = \int_{(0,1]} z \lambda(\dd z)$, so that there is no 
centering in \eqref{eq:supOU-def}. 
In particular, the supOU process is positive.

In Proposition \ref{prop:largexi} below we show that if $m_0(\pi) = \infty$, 
then the process $X(t)$ has a 
rather strange path behavior. If the jump measure $\lambda$ is 
unbounded, then the paths of the process $X(t)$ are almost surely unbounded on 
any nonempty interval, in particular, the process has no c\`adl\`ag modification.
Such pathological path 
behavior was observed in \cite{Maejima} for linear fractional stable motion.
On the other hand, if $\lambda$ is a stable L\'evy measure and 
$\pi$ is a finite measure such that $m_\delta(\pi) < \infty$
for some $\delta > 0$ then the process has c\`adl\`ag modification, see
\cite{basse2020sufficient}.

By $(\xi_k, \tau_k, \zeta_k)_{k \geq 0}$ we denote the points of 
the  Poisson random measure $\mu$. With this 
notation, if $\int_{(0,1]} z \lambda(\dd z) < \infty$,
\begin{equation} \label{eq:X+sum}
X(t) = \sum_{k: \tau_k \leq t} e^{-\xi_k (t - \tau_k)} \zeta_k.
\end{equation}

\begin{proposition} \label{prop:largexi}
Assume that $\int_{(0,1]} z \lambda(\dd z) < \infty$.
If $m_0(\pi) = \pi((0,\infty)) = \infty$, then
\[
\sup_{t \in [0,1]} X(t) \geq \sup \{ z : \overline \lambda(z) > 0 \}.
\]
In particular, if the jump sizes are unbounded, then 
\[
\sup_{t \in [0,1]} X(t) = \infty  \quad \text{a.s.}
\]
\end{proposition}

\begin{proof}
By \eqref{eq:X+sum} we have 
\[
\sup_{t \in [0,1]} X(t) \geq \sup \{ \zeta_k : \tau_k \in [0,1] \}.
\]
Since
\[
\pi \times \mathrm{Leb} \left( \{ (x, s ) : 0 \leq s \leq 1 \} \right)
= \pi ((0,\infty)) = \infty,
\]
there are infinitely many $k$'s such that $\tau_k < 1$. The 
corresponding $\zeta_k$'s are independent, thus the result follows.
\end{proof}

Therefore, in what follows we always assume that $m_0(\pi ) < \infty$,
otherwise all the expressions involving $\limsup X_t$ with 
unbounded jumps, would be infinite.

The main result of the paper is the following integral test.

\begin{theorem} \label{thm:as}
Assume that $\int_{(0,1]} z \lambda(\dd z) < \infty$ and 
$m_0(\pi)  < \infty$.
Let $f$ be a nondecreasing function tending to infinity, with $f(1) > 0$. 
If for some $K > 0$ 
\begin{equation} \label{eq:lambda-K-int} 
\int_1^\infty \overline \lambda(K f(t)) \dd t = \infty,
\end{equation}
then
\[
\limsup_{t \to \infty} \frac{X(t)}{f(t)} \geq K \quad \text{a.s.}
\]

For the converse assume that $f(1) > 0$,
$f$ is strictly increasing, differentiable,  and 
its inverse $f^{\leftarrow}$ is submultiplicative.
If for some $L > 0$
\begin{equation} \label{eq:eta-lambda-int} 
\int_1^\infty \left( \overline \lambda(L f(t)) + 
\overline \eta(L f(t)) \right) \dd t < \infty,
\end{equation}
then
\[
\limsup_{t \to \infty} \frac{X(t)}{f(t)} \leq 2 L \quad \text{a.s.}
\]
\end{theorem}

\begin{remark}
Observe that the almost sure growth rate does not depend
on the measure $\pi$, which determines the correlation function.
\end{remark}

\begin{remark} \label{remark:pos-neg}
From Theorem \ref{thm:as} the almost sure growth rate of a not necessarily 
positive supOU process with $\int_{[-1,1]} |z| \lambda(\dd z) < \infty$ 
can be deduced as follows.
The jump measure $\lambda$ can be written as a sum of a measure on
$(-\infty, 0)$ and a measure on $(0,\infty)$. This results in a decomposition of
the corresponding supOU process as 
$X(t)  =X_+ (t) - X_{-}(t)$, where $X_+$ and $X_-$ are independent 
positive supOU processes with jump measures 
$\lambda_+ (A) = \lambda( A \cap (0,\infty) )$ and 
$\lambda_-(A) = \lambda(-A \cap (-\infty, 0))$, for any Borel set $A \subset (0,\infty)$.
Then, for a nondecreasing function $f$ tending to infinity, with $f(1) > 0$, we have simply
\begin{equation}\label{eq:XandX+leq}
\limsup_{t \to \infty} \frac{X(t)}{f(t)} \leq \limsup_{t \to \infty} 
\frac{X_+(t)}{f(t)}.
\end{equation}
On the other hand, if $\limsup_{t \to \infty} \tfrac{X_+(t)}{f(t)} > K$, 
then there exists a random sequence $T_n \uparrow \infty$ independent of 
$X_{-}$ such that $\tfrac{X_+(T_n)}{f(T_n)} > K$.
Since $X_-$ is a stationary process and $T_n$ is independent of $X_{-}$,
the variables $(X_-(T_n))_n$ are identically distributed.
Using also that $T_n \uparrow \infty$ we obtain that $X_-(T_n) / f(T_n) \to 0$ 
in probability, which implies the almost sure convergence on a subsequence, that is
$\liminf_{n \to \infty} \tfrac{X_-(T_n)}{f(T_n)} = 0$ a.s. Therefore, 
$\limsup_{t \to \infty} \tfrac{X(t)}{f(t)} > K$. Since $K$ was arbitrary, together with \eqref{eq:XandX+leq} this implies that
\[
\limsup_{t \to \infty} \frac{X(t)}{f(t)} = \limsup_{t \to \infty} 
\frac{X_+(t)}{f(t)}.
\]
\end{remark}

\begin{proof}
First assume that the integral diverges. This already implies that 
$\overline \lambda(x) > 0$ for any $x > 0$.
Define
\[
A_n = \{ (x, s, z): z > K f(n), s \in (n-1, n], x>0 \}. 
\]
Then
\[
\nu(A_n) = \overline \lambda(K f(n)) m_0 (\pi).
\]
Since $\lim_{x \to \infty} f(x) = \infty$, $\overline \lambda$ is nonincreasing,
and $\nu(A_n) \to 0$, assumption \eqref{eq:lambda-K-int} implies that 
\[
\sum_{n = 1}^\infty \overline \lambda(K f(n) ) = \infty.
\]
The events $\{ \mu(A_n) \geq 1 \}$ are independent, and 
\[
\p ( \mu(A_n ) \geq 1) = 1- e^{-\nu(A_n)} \sim \nu(A_n).
\]
The second Borel--Cantelli lemma implies that almost surely 
$\mu(A_n) \geq 1$ infinitely often, thus
$\sup_{t \in (n-1,n]} X(t) > K f(n)$ infinitely often, implying
\[
\limsup_{t \to \infty} \frac{X(t)}{f(t)} \geq K.
\]

\smallskip 

For the converse, first note that for $t \in (n, n+1]$
\[
\begin{split}
X(t) & =\sum_{k: \tau_k \leq n } e^{-\xi_k (t - \tau_k)} \zeta_k + 
\sum_{k: \tau_k \in (n,t] } e^{-\xi_k (t - \tau_k)} \zeta_k \\
& \leq X(n)  + \sum_{k: \tau_k \in (n,n+1]} \zeta_k 
=: X(n) + Y_n.
\end{split}
\]
Then $(Y_n)_n$ is an iid sequence of infinitely divisible random variables
such that 
\[
\E e^{\ii \theta Y_1} = \exp \left\{ \int_{(0,\infty)} 
( e^{\ii \theta z } -1 ) m_0(\pi)  \lambda(\dd z) \right\},
\]
that is, its L\'evy measure is $m_0(\pi) \lambda$. Therefore, 
by an application of the Borel--Cantelli lemma, it is enough to show 
that
\begin{equation*} 
\sum_{n=1}^\infty \left[  \p ( X(0) > L f(n) )  +
\p ( Y_1 > L f(n) ) \right] < \infty. 
\end{equation*}
The latter is finite if and only if 
\[
\E f^{\leftarrow}(X(0)/L) + 
\E f^{\leftarrow}(Y_1/L)  < \infty.
\]
Applying \cite[Theorem 25.3]{Sato}, this is further equivalent to
\[
\int_{(1,\infty)} f^{\leftarrow}(z/L) \eta(\dd z) 
+ \int_{(1,\infty)} f^{\leftarrow}(z/L) \lambda(\dd z) 
< \infty.
\]
Integrating by parts and changing variables, we get that this is
equivalent to \eqref{eq:eta-lambda-int}.
\end{proof}

If $\overline \lambda \in \mathcal{D}$, then \eqref{eq:lambda-K-int} 
for some $K > 0$
implies \eqref{eq:lambda-K-int} for all $K > 0$, thus the limsup is infinite.
Similarly, if \eqref{eq:eta-lambda-int} is finite for some $L > 0$, 
then by Lemma \ref{lemma:domvar} it is finite for any $L > 0$, implying 
that the limsup is 0.
That is, there exists no proper normalization for the limit superior.
This is a common phenomena for maxima of heavy-tailed processes, see 
\cite[Theorem 3.5.1]{EKM} in the classical iid case, and 
\cite[Theorem III.13]{Bertoin} for subordinators.
In particular, if $\overline \lambda \in \mathcal{D}$ and 
\eqref{eq:lambda-growth}
holds, then we get a proper integral test for the growth rate.

\begin{corollary} \label{cor:supOU-as-D}
Assume that $\int_{(0,1]} z \lambda(\dd z) < \infty$, 
$m_0(\pi)  < \infty$, $\overline \lambda \in \mathcal{D}$, and 
\eqref{eq:lambda-growth}.
Assume that $f$ is a differentiable increasing function tending to infinity,
$f(1) > 0$, and its inverse $f^{\leftarrow}$ is submultiplicative. 
Then almost surely
\[
\limsup_{t \to \infty} \frac{X(t)}{f(t)} = 0 \quad \text{or } \quad \infty,
\]
according to whether 
\begin{equation}\label{e:cor2:test}
\int_1^\infty \overline \lambda( f(t)) \dd t < \infty 
\quad \text{or } \quad  = \infty.
\end{equation}
\end{corollary}

In particular, if $\overline \lambda \in \mathcal{RV}_{-\gamma}$, 
then almost surely $\limsup_{t \to \infty} \frac{X(t)}{t^a} = 0$ 
for $a>1/\gamma$ and $\limsup_{t \to \infty} \frac{X(t)}{t^a} = \infty$ 
for $a<1/\gamma$. For $a=1/\gamma$, 
the limit superior is zero or infinite according to whether 
$\ell(t^{1/\gamma}) t^{-1}$ is integrable or not, with $\ell$ being 
the slowly varying function in the representation of $\overline \lambda$,
i.e.~$\overline \lambda(t) = \ell(t) t^{-\gamma}$.

Dominated variation implies the heavy-tailed property, 
meaning that all the exponential moments are infinite. 
Theorem \ref{thm:as} is applicable for light-tailed L\'evy measures as 
well. Next we assume that the tail $\overline \lambda$ 
is exponential.

\begin{corollary} \label{cor:exp}
Assume that  $\int_{(0,1]} z \lambda(\dd z) < \infty$, 
$m_0(\pi)  < \infty$, and that for some $0 < c < c' < \infty$
\[
\int_{(1,\infty)} e^{c z} \lambda(\dd z) < \infty, \quad 
\int_{(1,\infty)} e^{c' z} \lambda(\dd z) = \infty.
\]
Then
\[
\frac{1}{c'} \leq \limsup_{t \to \infty} \frac{X(t)}{\log t} 
\leq \frac{2}{c} \quad \text{a.s.}
\]
\end{corollary}

\begin{proof}
Apply Theorem \ref{thm:as} with $f(t) = \log t$, 
for $t \geq e$, and 
$K = 1/c'$, $L= 1/c$, together with Corollary \ref{cor:Xmoments}.
\end{proof}

Finally, we consider the case when the L\'evy measure
$\lambda$ has bounded support. 
In this case Theorem \ref{thm:as} does not apply,
since the integral in \eqref{eq:lambda-K-int} is always finite.
Then the process has similar almost sure growth as iid Poisson random variables.

\begin{theorem} \label{thm:as-boundedlambda}
Assume that $\int_{(0,1]} z \lambda (\dd z) < \infty$, $m_0(\pi)<\infty$, and   
$ \inf \{ y : \overline \lambda(y) = 0 \} < \infty$. Then 
\[
0 < \limsup_{t \to \infty} \frac{\log \log t}{\log t} X(t) < \infty \quad \text{a.s.}
\]
\end{theorem}

\begin{proof}
In this case the moment generating function of $X(t)$ 
exists and, by \eqref{eq:eta-integral}
\begin{equation*}
\begin{split}
\Phi(t,s) & = \Phi(s) 
= \log  \E e^{s X(t)} = \int_{(0,\infty)} \left( e^{sy} - 1 \right)
\eta(\dd y) \\
& = m_{-1}(\pi) \int_0^\infty \overline \lambda(z) \frac{e^{sz} -1}{z} \dd z.
\end{split}
\end{equation*}
Using Markov's inequality
\begin{equation} \label{eq:Markov-bound}
\p ( X(t) > d) = \p (e^{s X(t)} > e^{sd} ) 
\leq \exp \left\{ \Phi(s) - sd \right\}.
\end{equation}

First assume that  $\lambda = \delta_{\{ 1 \}}$. We show that
\begin{equation} \label{eq:Poisson-as}
e^{-\varepsilon} \leq \limsup_{t \to \infty} \frac{\log \log t}{\log t}
X(t) \leq 2.
\end{equation}

As $s \to \infty$
\[
\Phi(s) = m_{-1}(\pi) \int_0^1 \frac{e^{sz} - 1}{z} \dd z = 
m_{-1}(\pi) \int_0^s \frac{e^{y} - 1}{y} \dd y \sim 
m_{-1}(\pi) \frac{e^{s}}{s}.
\]
Fix $\kappa > 1$ and 
choose $d = \kappa \log t / (\log \log t)$, and 
$s = \log \log t$. Then by \eqref{eq:Markov-bound},
for any $0 < \varepsilon < \kappa - 1$ if $t$ is large enough
\[
\p ( X(t) > d ) \leq t^{-(\kappa - \varepsilon)}.
\]
That is,
\[
\limsup_{n \to \infty} \frac{\log \log n }{\log n} X(n) \leq 1.
\]
Next, for $t \in (n,n+1]$
\[
 X(t) = \sum_{k: \tau_k \leq t} e^{- \xi_k (t - \tau_k)} \leq 
 X(n) + Y_n,
\]
where $Y_n = \sum_{k: \tau_k \in (n,n+1]} 1$. 
Then $(Y_n)$ is an iid sequence of Poisson random variables with 
parameter $m_0(\pi)$. 
For supremum of iid Poisson random variables we have
(see e.g.~Exercise 3.2.3, p.~65, in \cite{ChowTeicher})
\[
\limsup_{n \to \infty} \frac{\log \log n}{\log n} Y_n = 1,
\]
implying that almost surely
\[
\limsup_{t \to \infty} \frac{\log \log t}{\log t} X(t) \leq 2.
\]

To prove the lower bound, 
fix $\varepsilon > 0$ such that $\pi((0,\varepsilon)) > 0$. 
Put 
\[
A_n = \{ ( s, x, z) : s \in (n,n+1], x < \varepsilon, z > 0 \}.
\]
Clearly, $A_n$'s are disjoint, and 
$
\nu(A_n) = \pi((0,\varepsilon)) > 0.
$
Therefore $(\Lambda(A_n))_n$ is a sequence of iid Poisson random variables
with parameter $\pi((0,\varepsilon))$ and hence 
\[
\limsup_{n \to \infty} \frac{\log \log n}{\log n} \Lambda(A_n) = 1.
\]
Furthermore, 
\[
X(n+1) = \sum_{k : \tau_k \leq n+1} e^{-\xi_k (n+1- \tau_k)} \geq 
\sum_{ k : (\tau_k, \xi_k, 1) \in A_n } e^{-\xi_k} \geq 
\Lambda(A_n) e^{-\varepsilon},
\]
showing that 
\[
\limsup_{n \to \infty} \frac{\log \log n}{\log n} X(n) \geq 
e^{-\varepsilon}.
\]
Thus the proof of \eqref{eq:Poisson-as} is complete. 

The general case follows easily. If $\lambda$ is a finite measure, one 
can just replace all jumps by $K = \inf \{ y : \overline \lambda(y) = 0 \} < \infty$.
Thus we only have to handle the case 
$\lambda((0,1]) = \infty$, and $\overline \lambda(1) = 0$. Then
\[
\Phi(s) = m_{-1}(\pi) \int_0^1 \overline \lambda(z) \frac{e^{sz} -1}{z} \dd z
\leq m_{-1}(\pi) e^s \, \int_0^1 \overline \lambda(z) \dd z.
\]
Thus choosing $d = \kappa \log t /(\log \log t)$ and 
$s = \log \log t - \log \log \log t$ 
in \eqref{eq:Markov-bound} we obtain as before 
\[
\limsup_{n\to \infty} \frac{\log \log n}{\log n} X(n) \leq 1.
\]
For $t \in (n,n+1]$
\[
X(t) = \sum_{k: \tau_k \leq t} \zeta_k e^{-\xi_k (t - \tau_k)} \leq 
X(n) + Y_n,
\]
where 
\[
Y_n = \sum_{k: \tau_k \in (n,n+1]} \zeta_k.
\]
The assumption $\int_{(0,1]} z \lambda( \dd z) < \infty$ implies that 
$Y_n$'s are finite random variables. 
Then $Y_1,Y_2,\ldots$ are iid infinitely divisible random variables with 
L\'evy measure concentrated on $(0,1]$. Thus,  by Theorem 26.1 in \cite{Sato}
\[
\p (Y_1 > r ) = o(e^{-r \log r}) \quad \text{as } r \to \infty.
\]
Then Theorem 3.5.1 in \cite{EKM} shows that almost surely
\[
\limsup_{n \to \infty} \frac{\log \log n}{\log n} \max \{ Y_i: i=1,2,\ldots, n\}
\leq 1,
\]
and the upper bound follows.

The lower bound follows similarly as above. Pick $\varepsilon' > 0$ such that 
$\overline \lambda(\varepsilon') > 0$, and let 
\[
A_n' = \{ (s, x,z) : s \in (n,n+1), x \in (0,\varepsilon), z > \varepsilon' \}.
\]
Then $(\Lambda(A_n'))_n$ are iid Poisson random variables, and 
\[
X(n+1) \geq \Lambda(A_n') e^{-\varepsilon} \varepsilon'.
\]
\end{proof}

\noindent \textbf{Acknowledgement.}
We are grateful to the anonymous referees for helpful comments and suggestions.
DG was partially supported by the Croatian Science Foundation 
(HRZZ) grant Scaling in Stochastic Models (IP-2022-10-8081). 
PK was supported by the
J\'{a}nos Bolyai Research Scholarship of the Hungarian Academy of Sciences,
and by the NKFIH grant FK124141.


\end{document}